\documentclass[11pt]{amsart}
\usepackage{amsmath,amssymb,amsthm,upref,graphicx,mathrsfs}

\usepackage{color,xcolor}
\usepackage[
  colorlinks=true,
  linkcolor=blue,
  citecolor=blue,
  urlcolor=blue]{hyperref}

\numberwithin{equation}{section}

\newtheorem{thm}{Theorem}[section]
\newtheorem{cor}[thm]{Corollary}
\newtheorem{lem}[thm]{Lemma}
\newtheorem{prop}[thm]{Proposition}

\newtheorem{exaple}[thm]{Example}
\numberwithin{equation}{section}

\begin{document}

\leftline{ \scriptsize}

\vspace{1.3 cm}
\title
{ Inverse Eigenvalue Problem of Cell Matrices}
\author{ Sreyaun Khim$^{a}$ and Kijti Rodtes$^{b,\ast}$ }
\thanks{{\scriptsize
$^{\ast}$ Corresponding author.\\
E-mail addresses: sreyaun.khim@yahoo.com (Sreyaun Khim), kijtir@nu.ac.th (Kijti Rodtes).\\
$^{A}$ Department of Mathematics, Faculty of Science, Naresuan University, Phitsanulok 65000, Thailand.\\
$^{B}$ Department of Mathematics, Faculty of Science, Naresuan University, and Research Center for Academic Excellent in Mathematics, Phitsanulok 65000, Thailand.\\}}
\hskip -0.4 true cm

\maketitle


\begin{abstract}In this paper, we consider the problem of reconstructing an $n \times n$ cell matrix $D(\vec{x})$ constructed from a vector $\vec{x} = (x_{1}, x_{2},\dots, x_{n})$ of positive real numbers, from a given set of spectral data. In addition, we show that the spectrum of cell matrices $D(\vec{x})$ and $D(\pi(\vec{x}))$ are the same, for every permutation $\pi \in S_{n}$.
\end{abstract}

\vskip 0.2 true cm


\pagestyle{myheadings}
\markboth{\rightline {\scriptsize Sreyaun Khim and  Kijti Rodtes}}
         {\leftline{\scriptsize }}
\bigskip
\bigskip


\vskip 0.4 true cm

\section{Introduction}
An inverse eigenvalue problem (IEP) concerns the reconstruction of a matrix from given spectral data \cite{MT}.  Usually, there is some specific applications of IEP such as system and control theory, system identification, seismic tomography, principal component analysis, exploration and remote sensing, antenna array processing, geophysics, molecular spectroscopy, particle physics, structure analysis, circuit theory, mechanical system simulation, and so on \cite{MTC}.  The objective of an inverse eigenvalue problem is to construct matrices that maintain the specific structure as well as the given spectral property \cite{MT}.  There are many investigations about inverse eigenvalue problem of matrices.  For instance, an inverse eigenvalue problem for symmetric and normal matrices \cite{NR} was studied by Radwan.  An inverse eigenvalue problem for Jacobi matrices was studied by Wang and Zhong \cite{ZW}.  Also a solution of the inverse eigenvalue problem of certain singular Hermitian matrices was studied by Gyamfi, \cite{KB}.  Recently, in 2014 Nazari and Mahdinasab  worked on inverse eigenvalue problem of distance matrices by using orthogonal matrices technique and they constructed Euclidean distance matrices with the eigenvalue list, \cite{AM}.  They provided a new method for construction of distance matrices and also added some conditions so that they can get regular spherical matrices having the given eigenvalues. Even they considered providing a new algorithm to reconstruct distance matrices which are regular spherical matrices, however; they did not consider providing an algorithm to reconstruct cell matrices yet. Since the structure of cell matrices are much simpler than Euclidean distance matrices, it is theoretically interesting on finding the inverse eigenvalue problem on cell matrices.\\

Cell matrices were first introduced by Jacklic and Modic in 2010, \cite{GJ}, which are a special case of Euclidean distance matrices. In 2014, Tarazaga and Kurata, \cite{PT}, studied the set of cell matrices and its relationship with the cone of positive semidefinite diagonal matrices.  Recently in 2015, Kurata and Tarazaga \cite{HK}, consider the problem of finding cell matrices that is closest to a given Euclidean distance matrices with respect to the Frobenius norm.  They also discussed the majorization ordering of the eigenvalues of cell matrices.\\

In this study, we deal with the inverse eigenvalue problem on cell matrices for some lists of spectrums in which we use elementary matrices to find the spectrums of matrices.  We consider the reconstruction of $n \times n$ cell matrices with the given set of $2k$ eigenvalues in which we are free to choose $k$ distinct eigenvalues.  We also show that the spectrum of cell matrices $D(\pi(\vec{x}))$ have the same spectrum as cell matrices $D(\vec{x})$ for any $\pi \in S_{n}$.

\section{Preliminary}

An $n \times n$ matrix $D= (d_{ij})$ is said to be a \emph{Euclidean distance matrix (EDM)} if there exist $x_{1}, x_{2},\dots,{x_{n}}$ in some Euclidean space $\mathbb{R}^{r} (r\leq n)$, such that $d_{ij} = \lVert x_{i}- x_{j}\rVert_{2}^{2}$ for all $i, j= 1, 2,\dots, n$, where $\lVert .\rVert$ is the Euclidean norm. The following properties immediately hold according to the definition of EDM:
\begin{enumerate}
	\item $D$ is nonnegative matrix: $d_{ij}\geq 0$, for all $i,j= 1, 2, \dots, n,$
	\item $D$ is a symmetric matrix: $d_{ij}= d_{ji}$ for all $i, j= 1, \dots,n$,
	\item Diagonal elements are all zero: $d_{ii}= 0$ for all $i= 1, \dots, n$.
\end{enumerate}
\indent \indent These matrices were introduced by Menger in 1928, later they were studied by Schoenberg, \cite{IJ},  when studying positive definite functions and have been received considerable attention.  They are used in applications in geodesy, economics, genetics, psychology, biochemistry, engineering, etc \cite{AM}.  An Euclidean distance matrix $D$ is said to be \emph{spherical} if the construction points of $D$ lie on a hypersphere, otherwise, it said to be non-spherical.  Moreover, a spherical Euclidean distance matrix $D$ is \emph{regular} if the constructive points for $D$ lie on a hypersphere whose center coincides with the centroid of those points.

\indent Let $\vec{x} =(x_{i}), i = 1, 2, \dots, n$ be a vector of real numbers and $\vec{x} > 0.$ A \emph{cell matrix} $D \in \mathbb{R}^{n\times n}$ , associated with $\vec{x}$ denoted by $D(\vec{x})$, is defined as
\begin{equation*}
(D(\vec{x}))_{ij} = \left\{
\begin{array}{ll}
0 & \hbox{if $i = j$,} \\
x_{i}+x_{j} & \hbox{if $i \neq j$.} \\
\end{array}
\right.
\end{equation*}

\indent It is well known that cell matrices are Euclidean distance matrices \cite{PT}.  Furthermore, they are spherical EDM but need not be regular spherical EDM.  According to \cite{GJ} the determinant of principal sub-matrices of cell matrices $D(\vec{x})$ have the following form: \\
\begin{center}
$\det$ $D^{(i)} = (-1)^{i-1} 2^{i-2} (4(i-1) + \sum_{j=1}^{i}\sum_{l=1}^{j-1}\frac{(x_{j}- x_{l})^{2}}{x_{j} x_{l}}) \prod_{k=1}^{i} x_{k}$
,
\end{center}
where $D^{(i)}:= D(1: i, 1: i), i = 1, 2, \dots, n,$ are principal sub-matrices of a cell matrix.   Also, each cell matrix has exactly one positive eigenvalue, the rest of eigenvalues are negative, \cite{GJ}.

 We denote the spectrum of a square matrix $A$ by $\sigma (A)$ which is the set of all eigenvalues of $A$.  To determine the spectrum of matrices, we can use elementary matrices to simplify it.  Throughout this investigation, we use two types of elementary matrices $E$:
\begin{enumerate}
	\item $E=W_{ij}$, the row swapping $(R_{i}\leftrightarrow R_{j})$ matrix, (so $W_{ij}^{-1}= W_{ij}$,)
	\item $E=S_{ij}(\lambda)$, the row sum $(R_{i} \rightarrow R_{i} +\lambda R_{j})$ matrix, (so $S_{ij}^{-1} (\lambda)= S_{ij}(-\lambda)$).
\end{enumerate}

\indent It is well known that
\begin{center}
$\sigma(D(\vec{x})) = \sigma (ED(\vec{x})E^{-1})$,	
\end{center}
for any elementary matrices $E$.  Also, for the case of block matrices having the form $$A=  \left(
                                                                                        \begin{array}{cc}
                                                                                          X & Y \\
                                                                                          O & Z \\
                                                                                        \end{array}
                                                                                      \right),
$$ where $X,Z$ are square matrices and $O$ is the zero matrix, we have $\sigma(A)=\sigma(X)\cup\sigma(Z)$.  In particular, if we can express the $ED(\vec{x})E^{-1}$ as upper or lower triangular matrices then the entries on the main diagonal are eigenvalues of the matrix.

\section{A construction of cell matrices with $k$ distinct variables}
We first consider the inverse eigenvalue problem for $3\times 3$ cell matrices.  To do this, it is necessary to concentrate only on the list of spectrums with zero sum since their traces are all zero.
\begin{prop}
	Let $\sigma = \{\lambda_{1}, \lambda_{2}, \lambda_{3} \} \subseteq \mathbb{R}$ with $\lambda_{1}\geq 0 > \lambda_{3} \geq \lambda_{2}$ and $\lambda_{1}+ \lambda_{2}+ \lambda_{3}= 0.$ Then $\sigma$ is the spectrum of a cell matrix, $D(\vec{a})$, constructed by $\vec{a}= (\sqrt{\frac{\arrowvert \lambda_{1} \lambda_{2} \arrowvert}{2}}- \frac{\arrowvert \lambda_{3} \arrowvert}{2}, \frac{\arrowvert \lambda_{3}\arrowvert}{2}, \frac{\arrowvert \lambda_{3}\arrowvert}{2})$.
\end{prop}

\begin{proof}
Let $\vec{a}= (a_{1}, a_{2}, a_{3})$, where $a_{i}$ are positive real numbers. The explicit form of the cell matrix $D(\vec{a})$ is given by
\begin{center}
 $D(\vec{a})  =$
 $\left[
 \begin{array}{cccccccccccccccccc}
0 & a_{1}+a_{2} & a_{1}+a_{3}\\
a_{1}+a_{2} & 0 & a_{2}+a_{3}\\
a_{1}+a_{3} & a_{2}+ a_{3} & 0\\
 \end{array}
 \right]. $ 	
\end{center}
So, its characteristic polynomial is $$\Delta_ x (D(\vec{a}))= x^{3}- (\alpha^{2}+ \beta^{2}+ \gamma^{2}) x - 2 \alpha \beta \gamma, $$
where $\alpha= a_{1}+a_{2}, \beta= a_{1}+ a_{3}$ and $\gamma= a_{2}+ a_{3}$.  If we require that $\sigma$ is the spectrum of $D(\vec{a})$, then
$$\Delta_ x (D(\vec{a}))= x^{3} - (\lambda_{1}+ \lambda_{2}+ \lambda_{3}) x^{2} + (\lambda_{1} \lambda_{2}+ \lambda_{1}\lambda_{3} + \lambda_{2} \lambda_{3}) x - \lambda_{1}\lambda_{2}\lambda_{3}.$$
By comparing the coefficient of the polynomial $\Delta_ x (D(\vec{a}))$ and by using the assumption that $\lambda_{1}+\lambda_{2}+\lambda_{3}= 0,$ we conclude that
\begin{center}
$\alpha^{2}+ \beta^{2}+ \gamma^{2}=  \lambda_{1}^{2}- \lambda_{2} \lambda_{3}$\quad
and \quad $ 2\alpha \beta \gamma = \lambda_{1}\lambda_{2}\lambda_{3}.$
\end{center}
Now, we choose $\gamma= \arrowvert \lambda_{3}\arrowvert$, so $\gamma >0$.   Then $$2\alpha \beta = \frac{\lambda_{1} \lambda_{2} \lambda_{3}}{\gamma} = \frac{\lambda_{1} \lambda_{2} \lambda_{3}}{\arrowvert \lambda_{3} \arrowvert} = -\lambda_{1} \lambda_{2},$$
and thus

	\begin{eqnarray*}
		\alpha^{2} + \beta^{2}	&=& \lambda_{1}^{2}- \lambda_{2}\lambda_{3}- \gamma^{2}\\
		&=& \lambda_{1}^{2}- \lambda_{2}\lambda_{3}- \lambda_{3}^{2}\\
		&=& (\lambda_{1}- \lambda_{3}) (\lambda_{1}+ \lambda_{3}) - \lambda_{2}\lambda_{3}\\
		&=& (\lambda_{1}-\lambda_{3})(-\lambda_{2}) - \lambda_{2}\lambda_{3}\\
		&=& -\lambda_{1}\lambda_{2}\\
		&=& 2\alpha \beta .
	\end{eqnarray*}
 This implies that $\alpha = \beta$ which means $a_{1} +a_{2}= a_{1}+ a_{3}$ and hence $a_{2}= a_{3}$.  This also implies that $2 \alpha^{2} = -\lambda_{1} \lambda_{2}$ and hence $\alpha= \sqrt{\frac{\lvert \lambda_{1} \lambda_{2}\rvert}{2}}= \beta$.   Then
$$a_{1}= \sqrt{\frac{\lvert \lambda_{1} \lambda_{2}\rvert}{2}}- a_{2} \quad \hbox{and} \quad a_{2}= \frac{2a_{2}}{2}= \frac{a_{2}+a_{3}}{2}= \frac{\gamma}{2}= \frac{\lvert \lambda_{3}\rvert}{2}>0.$$
Since $\lvert \lambda_{3}\rvert = \min \{ \lvert \lambda_{1}\rvert, \lvert \lambda_{2}\rvert, \lvert \lambda_{3}\rvert \}$, $$a_{1}=  \sqrt{\frac{\lvert \lambda_{1} \lambda_{2}\rvert}{2}}- \frac{\lvert \lambda_{3}\rvert}{2}>0.$$ Therefore the cell matrix  $D(\vec{a})$  constructed from  $\vec{a}= (a_{1}, a_{2}, a_{2})$ has the spectrum $\sigma$.
\end{proof}

For example, if $\sigma= \{ -1, -2, 3\}$, then
\begin{center}
	$D(\vec{a})  =$
	$\left[
	\begin{array}{cccccccccccccccccc}
	0 & \sqrt{3} & \sqrt{3}\\
	\sqrt{3} & 0 & 1\\
	\sqrt{3} & 1 & 0\\
	\end{array}
	\right] $
\end{center}
is a cell matrix with the given spectrum $\sigma$.  It seems complicated to investigate the inverse eigenvalue problem in this way, even for $4\times 4$ matrices.  However, we have some observations on the reduction of the problem on cell matrices having a block form.

 Here, we present a sufficient condition for a reconstruction $n \times n$ cell matrices with the given set of $2k$ eigenvalues in which we are free to choose $k$ distinct eigenvalues.
\begin{thm}\label{mainthm}
	Let $S= \{\lambda_{1}, \lambda_{2},\dots,\lambda_{k},\lambda_{k+1},\dots,\lambda_{k+1},\dots,\lambda_{2k},\dots, \lambda_{2k} \}$ be a subset of real numbers in which $\lambda_{k+1}, \dots,  \lambda_{2k}$ have multiplicities $(l_{1}-1), \dots,  (l_{k}-1)$ respectively, where $l_{1}  + \dots + l_{k} = n$ and $l_{i} >1$ for $i = 1, 2, \dots , k$.  If elements of $S$ satisfy:
	\begin{enumerate}
		\item$\lambda_{1}>0$ and $\lambda_{2}, \lambda_{3}, \dots, \lambda_{k} ,\dots ,\lambda_{2k} < 0$,
		\item$ \lambda_{1}, \lambda_{2},\dots, \lambda_{k},$ are the roots of the characteristic polynomial of the $k \times k$ matrix $D^{(k)}(\vec{x})$ with
	\end{enumerate}
\begin{equation*}
(D^{(k)}(\vec{x}))_{ij} = \left\{
\begin{array}{ll}
(l_{k-j+1}-1)(2x_{k-j+1}) & \hbox{if $i = j$,} \\
l_{k-j+1}(x_{k-j+1}+x_{k-i+1}) & \hbox{if $i \neq j$,} \\
\end{array}
\right.
\end{equation*}
where $x_{i}:= -\frac{\lambda_{k+i}}{2},$ for $i=1,2, \dots, k$, then there is a cell matrix $D(\vec{x})$ such that $\sigma(D(\vec{x})) = S$.  Explicitly, $D(\vec{x})$ can be reconstructed from the vector
	$$ \vec{x} = (\underbrace{-\frac{\lambda_{k+1}}{2},  \dots , -\frac{\lambda_{k+1}}{2}}_{l_{1}},\underbrace{-\frac{\lambda_{k+2}}{2}, \dots , -\frac{\lambda_{k+2}}{2}}_{l_{2}}, \dots, \underbrace{-\frac{\lambda_{2k}}{2},  \dots , -\frac{\lambda_{2k}}{2}}_{l_{k}}).$$
	\end{thm}

\begin{proof}
	We construct $D(\vec{x})$ from the vector $\vec{x}$ given in the theorem.
 Denote $1_{m \times n}$ the $m \times n$ matrix whose its entry are all $1$ and $x_{i}:= -\frac{\lambda_{k+i}}{2}$ for each $i= 1, 2, \dots, k$.  Then,  $D(\vec{x})$ is a block matrix whose the $(s,t)$ block is given by the $l_{s} \times l_{t}$ matrix 
\begin{equation*}
D^{st} = \left\{
\begin{array}{ll}
(x_{s}+x_{t})1_{l_{s}\times l_{t}}& \hbox{if $s \neq t $,} \\
2x_{s}(1_{l_{s}\times l_{t}}- I_{l_{s}}) & \hbox{if $s = t$,} \\
\end{array}
\right.
\end{equation*}
for $1\leq s, t \leq k$.

 For the first step, we use row sums and column sums on the rows $n, n-1, \dots, n-l_{k}+2$ by the row $n-l_{k}+1:=j_{k}$; namely, we calculate
\begin{center}
 $S_{n-l_{k}+2,j_{k}}(-1) \cdots S_{n-1,j_{k}}(-1)S_{n,j_{k}}(-1)D(\vec{x}) S^{-1}_{n,j_{k}}(-1) S^{-1}_{n-1,j_{k}}(-1) \cdots S^{-1}_{n-l_{k}+2,j_{k}}(-1) .$
 \end{center}
Then	$$D(\vec{x})  \sim
	\left[
	\begin{array}{cccccc}
D^{1}(\vec{x}) & \ast\\
0 & D(\lambda_{2k})\\
	\end{array}
	\right], $$ where $D(\lambda_{2k}) =\operatorname{diag} (\underbrace{\lambda_{2k},\dots,\lambda_{2k}}_{l_{k}-1})$ and
\begin{center}
 $D^{1}(\vec{x})  =$
	$\left[
	\begin{array}{cccccccccccccccccc}
	D^{11}& D^{12} & \dots & D^{1(k-1)} & \vec{v}_{1k}\\
	D^{12}& D^{22} & \dots & D^{2(k-1)} & \vec{v}_{2k}\\
	\vdots& \vdots & \ddots & \vdots & \vdots\\
	D^{1(k-1)}& D^{2(k-1)} & \dots & D^{(k-1)(k-1)} & \vec{v}_{(k-1)k}\\
	\vec{w}^{T}_{1k}&\vec{w}^{T}_{2k}& \dots& \vec{w}^{T}_{(k-1)(k)}& (l_{k}-1)(2x_{k})\\
	\end{array}
	\right] $
	\end{center}	
in which $\vec{v}_{ik}= (\underbrace{l_{k}(x_{i}+x_{k}), \dots, l_{k}(x_{i}+x_{k}) }_{l_{i}})^{T}$ and $\vec{w}_{ik}=(\underbrace{(x_{i}+x_{k}, \dots, (x_{i}+x_{k})}_{l_{i}})^{T}$.
Next, we swap the row $j_{k}$ with the  row $n-l_{k}-l_{k-1}+1:= j_{k-1}$ of $D^{1}(\vec{x})$; namely, we calculate   $$W_{j_{k},j_{k-1}}D^{1}(\vec{x})W_{j_{k},j_{k-1}} .$$
Then $D^{1}(\vec{x})  \sim D^{11}(\vec{x})$, where
\begin{center}
$D^{11}(\vec{x})=\left[
\begin{array}{cccccccccccccccccc}
D^{11}& D^{12}& \dots & D^{1(k-2)} & \vec{v}_{1k} & D^{1(k-1)}\\

D^{21}& D^{22}& \dots & D^{2(k-2)} & \vec{v}_{2k} & D^{2(k-1)}\\

\vdots & \vdots& \ddots& \vdots& \vdots&\vdots\\

D^{1(k-2)}& D^{2(k-2)}& \dots & D^{(k-2)(k-2)} & \vec{v}_{(k-2)k} & D^{(k-2)(k-1)}\\

\vec{w}^{T}_{1k}&\vec{w}^{T}_{2k}& \dots& \vec{w}^{T}_{(k-2)k}& (l_{k}-1)(2x_{k}) & \vec{w}^{T}_{(k-1)k}\\

D^{1(k-1)}& D^{2(k-1)}& \dots & D^{(k-2)(k-1)} & \vec{v}_{(k-1)k} & D^{(k-1)(k-1)}\\
\end{array}
\right].$
\end{center}

For the second step, we use row sums and column sums on the row $j_{k}, j_{k}-1, \dots, j_{k-1}+3, j_{k-1}+2$ by the row $j_{k-1}+1$ on the matrix $D^{11}(\vec{x})$; namely, we calculate
\begin{center}
$S_{j_{k-1}+2, j_{k-1}+1}(-1)S_{j_{k-1}+3, j_{k-1}+1}(-1) \cdots S_{j_{k}-1, j_{k-1}+1}(-1) S_{j_{k},j_{k-1}+1}(-1)$ \\
 $D^{11}(\vec{x})S^{-1}_{j_{k},j_{k-1}+1}(-1)  S^{-1}_{j_{k}-1, j_{k-1}+1}(-1) \cdots S^{-1}_{j_{k-1}+3, j_{k-1}+1}(-1) S^{-1}_{j_{k-1}+2, j_{k-1}+1}(-1). $
\end{center}
Then
$$D^{11}(\vec{x})  \sim
\left[
\begin{array}{cccccc}
D^{2}(\vec{x}) & \ast\\
0 & D(\lambda_{2k-1})\\
\end{array}
\right], $$ where $D(\lambda_{2k-1}) = \operatorname{diag} (\underbrace{\lambda_{2k-1},\dots,\lambda_{2k-1}}_{l_{k-1}-1})$ and
\begin{center}
$D^{2}(\vec{x})= \left[
		\begin{array}{ccccccccccccc}
		D^{11}& D^{12}& \dots & D^{1(k-2)} & \vec{v}_{1k} & \vec{v}_{1(k-1)}\\
		D^{12}& D^{22}& \dots & D^{2(k-2)} & \vec{v}_{2k} & \vec{v}_{2(k-1)}\\
		\vdots & \vdots& \ddots & \vdots & \vdots & \vdots\\
		D^{1(k-2)}& D^{2(k-2)}& \dots & D^{(k-2)(k-2)} & \vec{v}_{(k-2)k} & \vec{v}_{(k-2)(k-1)}\\
		\vec{w}^{T}_{1k}&\vec{w}^{T}_{2k}& \dots &\vec{w}^{T}_{(k-2)k}& (l_{k}-1)(2x_{k})& l_{k-1}(x_{k-1}+x_{k})\\
	\vec{w}^{T}_{1(k-1)}&\vec{w}^{T}_{2(k-1)}& \dots &\vec{w}^{T}_{(k-2)(k-1)}& l_{k}(x_{k-1}+x_{k})& (l_{k-1}-1)(2x_{k-1})\\	
		\end{array}
		\right]$
\end{center}
in which 
$$\vec{v}_{i(k-1)}= (\underbrace{l_{k-1}(x_{i}+x_{k-1}), \dots, l_{k-1}(x_{i}+x_{k-1}) }_{l_{i}})^{T} \hbox{ and } \vec{w}_{i(k-1)}= (\underbrace{x_{i}+x_{k-1}, \dots, x_{i}+x_{k-1}}_{l_{i}})^{T}.$$
Next, we swap the row $l_{1}+ \dots+ l_{k-3}+1= j_{k-2}$ with the row $j_{k-1}$ and swap the row $j_{k-2}+1$ with the row $j_{k-1}+1$; namely, we calculate $$W_{j_{k-2}+1, j_{k-1}+1}W_{j_{k-2}, j_{k-1}} D^{2}(\vec{x})W_{j_{k-2}, j_{k-1}}W_{j_{k-2}+1, j_{k-1}+1}.$$   Then $D^{2}(\vec{x})\sim D^{22}(\vec{x})$, where $D^{22}(\vec{x})$ is the matrix
\begin{center}
	$ \left[
	\begin{array}{ccccccccccccc}
	D^{11}& D^{12}& \dots & D^{1(k-3)} & \vec{v}_{1k} & \vec{v}_{1(k-1)} & D^{1(k-2)}\\
	
	D^{12}& D^{22}& \dots & D^{2(k-3)} & \vec{v}_{2k} & \vec{v}_{2(k-1)} & D^{2(k-2)}\\
	
	\vdots & \vdots& \ddots & \vdots & \vdots & \vdots & \vdots\\
	
	D^{1(k-3)}& D^{2(k-3)}& \dots & D^{(k-3)(k-3)} & \vec{v}_{(k-3)k} & \vec{v}_{(k-3)(k-1)} & D^{(k-3)(k-3)}\\
	
	\vec{w}^{T}_{1k}&\vec{w}^{T}_{2k}& \dots &\vec{w}^{T}_{(k-3)k}& (l_{k}-1)(2x_{k})& l_{k-1}(x_{k-1}+x_{k}) & \vec{w}^{T}_{(k-2)k}\\
	
	\vec{w}^{T}_{1(k-1)}&\vec{w}^{T}_{2(k-1)}& \dots &\vec{w}^{T}_{(k-3)(k-1)}& l_{k}(x_{k-1}+x_{k})& (l_{k-1}-1)(2x_{k-1}) & \vec{w}^{T}_{(k-2)(k-1)}\\	
	
	D^{1(k-2)}& D^{2(k-2)}& \dots & D^{(k-3)(k-3)} & \vec{v}_{(k-2)k} & \vec{v}_{(k-2)(k-1)} & D^{(k-2)(k-2)}\\
	
	\end{array}
	\right]$
\end{center}
	
We continue the same process until the step $k^{th}$ which we use row sums and column sums on the row $l_{1}, l_{1}-1, \dots, 3, 2$ by the row $ 1$ on the matrix $D_{k-1}:=D^{((k-1)(k-1))}(\vec{x})$; namely, we calculate
\begin{center}
$S_{2,1}(-1)S_{3,1}(-1)\cdots S_{l_{1}-1,1}(-1)S_{l_{1},1}(-1) D_{(k-1)}S^{-1}_{l_{1},1}(-1)S^{-1}_{l_{1}-1,1}(-1) \cdots S^{-1}_{3,1}(-1) S^{-1}_{2,1}(-1) $.
\end{center}
Then
$$D_{k-1}  \sim
\left[
\begin{array}{cccccc}
D^{k}(\vec{x}) & \ast\\
0 & D(\lambda_{k+1})\\
\end{array}
\right], $$ where $D(\lambda_{k+1}) = \operatorname{diag} (\underbrace{\lambda_{k+1},\dots,\lambda_{k+1}}_{l_{1}-1})$ and
\begin{center}
	$D^{k}(\vec{x})= \left[
	\begin{array}{cccccccccccc}
(l_{k}-1)(2x_{k}) & l_{k-1}(x_{k-1} + x_{k}) & \dots & l_{2} (x_{2}+x_{k}) & l_{1} (x_{1} + x_{k})\\

l_{k} (x_{k-1}+ x_{k}) & (l_{k-1}-1)(2x_{k-1}) &  \dots & l_{2} (x_{2}+x_{k-1}) & l_{1} (x_{1} + x_{k-1})\\

l_{k}(x_{k-2}+x_{k}) & l_{k-1}(x_{k-2} + x_{k-1}) &  \dots & l_{2} (x_{2}+x_{k-2}) & l_{1} (x_{1} + x_{k-2})\\

\vdots & \vdots &  \ddots & \vdots & \vdots\\

l_{k}(x_{2}+x_{k}) & l_{k-1}(x_{2} + x_{k-1}) &  \dots & (l_{2}-1) (2x_{2}) & l_{1} (x_{1} + x_{2})\\

l_{k}(x_{1}+x_{k}) & l_{k-1}(x_{1} + x_{k-1}) &  \dots & l_{2} (x_{1}+ x_{2}) & (l_{1}-1) (2x_{1})\\	
	\end{array}
	\right]_{k\times k.}$
\end{center}

If  $ \lambda_{1}, \lambda_{2},\dots, \lambda_{k}$ are the roots of the characteristic polynomial of the matrix $D^{(k)}(\vec{x})$, then $D(\vec{x})$ is a cell matrix with the given spectrum $S$.
\end{proof}
In particular, when $k=1$, we have:
\begin{cor} Let $n$ be a positive integer and $\lambda$ be a positive real number.  If $$S=\{ (n-1)\lambda, \underbrace{-\lambda,-\lambda,\dots,-\lambda}_{n-1} \},$$ then there is an $n\times n$ cell matrix $D(\vec{x})$ with $\sigma(D(\vec{x}))=S$.  Precisely $D(\vec{x})$ is constructed from the vector $ \vec{x} =(\underbrace{\frac{\lambda}{2}, \frac{\lambda}{2}, \cdots ,  \frac{\lambda}{2}}_n).$
\end{cor}
Also, when $k=2$ we have:
\begin{cor}\label{cor2}
	If $ S= \{\lambda_{1}, \lambda_{2}, \underbrace{\lambda_{3},\lambda_{3},\dots,\lambda_{3}}_{l_{1}-1},\underbrace{\lambda_{4},\lambda_{4},\dots,\lambda_{4}}_{l_{2}-1}\}$ with $l_{1}+ l_{2}= n$ and $l_{1}, l_{2}> 0$ satisfies:
	\begin{enumerate}
		\item$\lambda_{1}>0 $ and $ \lambda_{2},\lambda_{3}, \lambda_{4}< 0 $,
	    \item$ \lambda_{1}=[(l_{1}-1)(-\frac{\lambda_{3}}{2})+ (l_{2}-1)(-\frac{\lambda_{4}}{2})]+$ \\
		\indent \indent $\sqrt{(l_{1}(n-2)+1)(-\frac{\lambda_{3}}{2})^{2}+ \frac{1}{2}(n-1)\lambda_{3}\lambda_{4}+(n^{2}-n(l_{1}+2)+2l_{1}+1)(-\frac{\lambda_{4}}{2})^{2}}$,
		\item$  \lambda_{2}=[(l_{1}-1)(-\frac{\lambda_{3}}{2})+ (l_{2}-1)(-\frac{\lambda_{4}}{2})]-$ \\
		\indent \indent	$\sqrt{(l_{1}(n-2)+1)(-\frac{\lambda_{3}}{2})^{2}+ \frac{1}{2}(n-1)\lambda_{3}\lambda_{4}+(n^{2}-n(l_{1}+2)+2l_{1}+1)(-\frac{\lambda_{4}}{2})^{2}}$.
	\end{enumerate}
	then there is a cell matrix $D(\vec{x})$ such that $\sigma (D(\vec{x})) = S$. Explicitly, $D(\vec{x})$ can be reconstructed from the vector
	\begin{center}
		$ \vec{x} =(\underbrace{-\frac{\lambda_{3}}{2}, -\frac{\lambda_{3}}{2}, \cdots , -\frac{\lambda_{3}}{2}}_{l_{1}},\underbrace{-\frac{\lambda_{4}}{2}, -\frac{\lambda_{4}}{2}, \cdots , -\frac{\lambda_{4}}{2}}_{l_{2}}).$
	\end{center}
\end{cor}

\begin{proof}
By Theorem \ref{mainthm}, for the case $k=2$, we have
$$ D^{2}(\vec{x}) =\left[
\begin{array}{ccccccccc}
-(l_{2}-1)\lambda_{4}& -\frac{l_{1}}{2}(\lambda_{3}+\lambda_{4})\\
-\frac{l_{2}}{2}(\lambda_{3}+\lambda_{4}) & -(l_{1}-1)\lambda_{3}
\end{array}
\right].$$
So, $\sigma(D^{2}(\vec{x}))=\{\lambda_1, \lambda_2\}$.
\end{proof}
For example, if we choose $\lambda_{3}= -2$, $\lambda_{4}= -4$ and $l_1=5, l_2=6$, then we compute by Corollary \ref{cor2} that $\lambda_1=14+ \sqrt{306}$ and $\lambda_2= 14- \sqrt{306}$.  Then, $$\sigma=\{14+ \sqrt{306}, 14- \sqrt{306}, -2, -2, -2, -2, -4, -4, -4, -4, -4\}$$ is the spectrum of the $11\times 11$ cell matrix constructed from the vector $$\vec{x}= (1, 1, 1, 1, 1, 2, 2, 2, 2, 2, 2).$$	 Precisely, 	
	\begin{center}
		$D(\vec{x}) =$  	$ \left[
		\begin{array}{ccccccccccccc}
		0 & 2& 2& 2& 2& 3& 3& 3& 3 &3 &3\\
		2 & 0& 2& 2& 2& 3& 3& 3& 3 &3 &3\\
		2 & 2& 0& 2& 2& 3& 3& 3& 3 &3 &3\\
		2 & 2& 2& 0& 2& 3& 3& 3& 3 &3 &3\\
		2 & 2& 2& 2& 0& 3& 3& 3& 3 &3 &3\\
		3 & 3& 3& 3& 3& 0& 4& 4& 4 &4 &4\\
		3 & 3& 3& 3& 3& 4& 0& 4& 4 &4 &4\\
		3 & 3& 3& 3& 3& 4& 4& 0& 4 &4 &4\\
		3 & 3& 3& 3& 3& 4& 4& 4& 0 &4 &4\\
		3 & 3& 3& 3& 3& 4& 4& 4& 4 &0 &4\\
		3 & 3& 3& 3& 3& 4& 4& 4& 4 &4 &0\\
		\end{array}
		\right]_{(11 \times 11)}$
	\end{center}
	is a cell matrix with the spectrum $\sigma$.

\begin{exaple}
	We construct a cell matrix of order $13$ by using Theorem \ref{mainthm}.   Here, we choose $\lambda_{4}= -2,\lambda_{5}= -3,\lambda_{6}= -5$ and choose $l_{1}= 4, l_{2}= 4,l_{3}= 5$.  Then the eigenvalues of $D^{(3)}(\vec{x})$ are $20+\sqrt{511}, 20- \sqrt{511}$ and $-5$.  So $$\sigma=\{ 20+\sqrt{511}, 20- \sqrt{511} , -5, -2, -2, -2, -3, -3, -3, -5, -5, -5, -5\}$$ satisfies the condition of Theorem \ref{mainthm}.  Thus we may construct the cell matrix from the vector
		$$\vec{x}= (1, 1, 1, 1, \frac{3}{2},\frac{3}{2}, \frac{3}{2}, \frac{3}{2}, \frac{5}{2}, \frac{5}{2}, \frac{5}{2}, \frac{5}{2}, \frac{5}{2}).$$	
In fact,
	\begin{center}
		$D(\vec{x}) =$  	$ \left[
		\begin{array}{cccccccccccccc}
		0 & 2 & 2 & 2 & \frac{5}{2} & \frac{5}{2} & \frac{5}{2} & \frac{5}{2} & \frac{7}{2} & \frac{7}{2} & \frac{7}{2} & \frac{7}{2} & \frac{7}{2} \\
        2 & 0 & 2 & 2 & \frac{5}{2} & \frac{5}{2} & \frac{5}{2} & \frac{5}{2} & \frac{7}{2} & \frac{7}{2} & \frac{7}{2} & \frac{7}{2} & \frac{7}{2} \\
        2 & 2 & 0 & 2 & \frac{5}{2} & \frac{5}{2} & \frac{5}{2} & \frac{5}{2} & \frac{7}{2} & \frac{7}{2} & \frac{7}{2} & \frac{7}{2} & \frac{7}{2} \\
        2 & 2 & 2 & 0 & \frac{5}{2} & \frac{5}{2} & \frac{5}{2} & \frac{5}{2} & \frac{7}{2} & \frac{7}{2} & \frac{7}{2} & \frac{7}{2} & \frac{7}{2} \\
        \frac{5}{2} & \frac{5}{2} & \frac{5}{2} & \frac{5}{2} & 0 & 3 & 3 & 3 & 4 & 4 & 4 & 4 & 4\\
        \frac{5}{2} & \frac{5}{2} & \frac{5}{2} & \frac{5}{2} & 3 & 0 & 3 & 3 & 4 & 4 & 4 & 4 & 4\\
        \frac{5}{2} & \frac{5}{2} & \frac{5}{2} & \frac{5}{2} & 3 & 3 & 0 & 3 & 4 & 4 & 4 & 4 & 4\\
        \frac{5}{2} & \frac{5}{2} & \frac{5}{2} & \frac{5}{2} & 3 & 3 & 3 & 0 & 4 & 4 & 4 & 4 & 4\\
        \frac{7}{2} & \frac{7}{2} & \frac{7}{2} & \frac{7}{2} & 4 & 4 & 4 & 4 & 0 & 5 & 5 & 5 & 5\\
        \frac{7}{2} & \frac{7}{2} & \frac{7}{2} & \frac{7}{2} & 4 & 4 & 4 & 4 & 5 & 0 & 5 & 5 & 5\\
        \frac{7}{2} & \frac{7}{2} & \frac{7}{2} & \frac{7}{2} & 4 & 4 & 4 & 4 & 5 & 5 & 0 & 5 & 5\\
        \frac{7}{2} & \frac{7}{2} & \frac{7}{2} & \frac{7}{2} & 4 & 4 & 4 & 4 & 5 & 5 & 5 & 0 & 5\\
        \frac{7}{2} & \frac{7}{2} & \frac{7}{2} & \frac{7}{2} & 4 & 4 & 4 & 4 & 5 & 5 & 5 & 5 & 0\\
\end{array}
		\right]_{(13 \times 13)}$
	\end{center}
	is a cell matrix with the spectrum $\sigma$.	
\end{exaple}

Note that the $k\times k$ matrix $D^{(k)}(\vec{x})$ in the Theorem \ref{mainthm} is a positive real matrix which may not be symmetric matrix depending on $l_i$'s and $\lambda_{k+i}$'s. However, since its spectrum is $\{\lambda_1,\dots,\lambda_k\}$ which is a subset of the spectrum of the cell matrix $D(\vec{x})$ with $\lambda_1>0$,  it  always happens that the spectrum of $D^{(k)}(\vec{x})$ must be a subset of real numbers with only one positive eigenvalue $\lambda_1$ satisfying $\lambda_1> |\lambda_2|+\cdots+|\lambda_k|$.

\section{Invariance of the spectrum}
For a vector $\vec{x}=(x_1,x_2,\dots,x_n)\in \mathbb{R}^n$ and a permutation $\pi\in S_n$, we denote $\pi(\vec{x})$ the vector $(x_{\pi(1)},x_{\pi(2)},\dots,x_{\pi(n)})$. 
\begin{lem}\label{lemforse}
Let $D(\vec{x})$ be the cell matrix constructed from $\vec{x}= (x_{1}, x_{2}, \dots, x_{n})$. Let $\pi_{1}= (l, k),$ for some distinct $l, k \in \{1, 2, \dots, n\}$, be a transposition in $S_{n}$. If $P$ is the permutation matrix corresponding to $\pi_{1}$, then $PD(\pi_{1}(\vec{x}))P= D(\vec{x})$.
\end{lem}

\begin{proof}
Since $\pi_{1}$ is a transposition $(l, k)$,

\begin{equation*}
(\pi_{1}(\vec{x}))_{i}= \left\{
\begin{array}{ll}
x_{i} & \hbox{if $i \neq l, k $} \\
x_{l}&\hbox{if $i =k $} \\
x_{k}&\hbox{if $i = l$.} \\
\end{array}
\right.
\end{equation*}
Note also that $PD(\pi_{1}(\vec{x}))P$ is the matrix obtained from $D(\pi_{1}(\vec{x}))$ by swapping the $l$-th row with the $k$-th row and the $l$-th column with the $j$-th column. Then, we can list the entries of $D(\vec{x}), D(\pi_{1}(\vec{x}))$ and $P(D(\pi_{1}(\vec{x})))P$ as in the table below:
\begin{center}
	\begin{tabular}{l|c|c|c}
		\hline
		$(i, j)$ & $D_{ij}= (D(\vec{x}))_{ij}$ & $\tilde{D}_{ij} = (D(\pi_{1}(\vec{x})))_{ij}$ &$ (P\tilde{D}P)_{ij}$ \\
		\hline
		$i=j$ & $0$ & $0$ & $0$ \\
		$i=l, j=k$ & $x_{l}+x_{k}$ & $ x_{k}+x_{l}  $ & $ \tilde{D}_{kl}= x_{l}+x_{k}$ \\
		$i = k, j=l $ & $x_{k}+x_{l}$ & $x_{l}+x_{k}$ & $ \tilde{D}_{lk}= x_{k}+x_{l}$ \\
		$i\notin \{l,k\}, j=k$ & $x_{i}+x_{k}$ & $x_{i}+x_{l}$ & $ \tilde{D}_{il} =x_{i}+x_{k}$ \\
		$i\notin \{l,k\}, j=l$ & $x_{i}+x_{l}$ & $x_{i}+x_{k}$ & $ \tilde{D}_{ik}= x_{i}+x_{l}$ \\
		$i = l, j\notin \{l,k\}$ & $x_{l}+x_{j}$ & $x_{k}+x_{j}$ & $ \tilde{D}_{kj}= x_{l}+x_{j}$ \\
		$i = k, j\notin \{l,k\}$ & $x_{k}+x_{j}$ & $x_{l}+x_{j}$ & $ \tilde{D}_{lj}= x_{k}+x_{j}$ \\
		$i, j\notin \{l,k\}$ & $x_{i}+x_{j}$ & $x_{i}+x_{j}$ & $ \tilde{D}_{ij}= x_{i}+x_{j}$ \\		
		\hline
	\end{tabular}
\end{center}
Hence $PD(\pi_{1}(\vec{x}))P= D(\vec{x})$.
\end{proof}

\begin{thm}
	Let $D(\vec{x})$ be $n \times n$ cell matrices with $ \vec{x} =(x_{1}, x_{2}, \dots, x_{n}).$ If $\pi \in S_{n}$ then $D(\vec{x})$ and $D(\pi(\vec{x}))$ have the same spectrum.	
\end{thm}

\begin{proof}
Let  $\pi \in S_{n}$. Note that $\pi$ can be written as a composition of transpositions in $S_{n}$; say,
\begin{center}
	$\pi= \pi_{m} \circ \pi_{m-1} \circ \dots \circ \pi_{1},$
\end{center}
where $\pi_{i}$ is a transposition in $S_{n}$, for each $i= 1, 2, \dots, m.$ We denote $\vec{x}_{i}= \pi_{i}(\vec{x}_{i-1})$, for $i= 1, 2, \dots, m,$ and $\vec{x}_{0}= \vec{x}$. So, $\vec{x}_{m}= \pi(\vec{x})$. Let $P_{i}$ be the permutation matrix corresponding to the transposition $\pi_{i}$. By Lemma \ref{lemforse}, we have $$P_{i}(D(\pi_{i}(\vec{x}_{i-1})))P_{i}= D(\vec{x}_{i-1}).$$ Since $P_{i}^{-1} = P_{i},$ we have that $D(\pi_{i}(\vec{x}_{i-1}))$ is similar to $D(\vec{x}_{i-1})$.
Hence $\sigma(D(\pi_{i}(\vec{x}_{i-1})))= \sigma(D(\vec{x}_{i-1}))$, for each $i= 1, 2, \dots,m.$ Therefore\\
	\begin{eqnarray*}
		\sigma(D(\pi(\vec{x})))	&=& \sigma(D(\vec{x}_{m}))\\
		&=& \sigma(D(\vec{x}_{m-1}))\\
		&& \vdots \\
		&=& \sigma(D(\vec{x}_{0}))= \sigma(D(\vec{x})).
	\end{eqnarray*}	
Hence $\sigma(D(\pi(\vec{x}))) = \sigma(D(\vec{x}))$.
\end{proof}
For example, if $\vec{x}=(1,2,3,4,5,6,7)$ and $\pi=(1\,4)(2\,5)(3\,7\,6)\in S_3$, then the vector $\pi (\vec{x})=(4,5,7,1,2,3,6)$.  Hence, the cell matrices 
$$  D(\vec{x})=\left(
                 \begin{array}{ccccccc}
                   0 & 3 & 4& 5 & 6 & 7 & 8 \\
                   3 & 0& 5 & 6 & 7 & 8 & 9 \\
                   4 & 5 & 0 & 7 & 8 & 9 & 10 \\
                   5 & 6 & 7 & 0 & 9 & 10 & 11 \\
                   6 & 7 & 8 & 9 & 0 & 11 & 12 \\
                   7 & 8 & 9 & 10 & 11 & 0 & 13 \\
                   8 & 9 & 10 & 11 & 12 & 13 & 0 \\
                 \end{array}
               \right)
\hbox { and }  D(\pi(\vec{x}))=\left(
                 \begin{array}{ccccccc}
                   0 & 9 & 11& 5 & 6 & 7 & 10 \\
                   9 & 0& 12 & 6 & 7 & 8 & 11 \\
                   11 & 12 & 0 & 8 & 9 & 10 & 13 \\
                   5 & 6 & 8 & 0 & 3 & 4 & 7 \\
                   6 & 7 & 9 & 3 & 0 & 5 & 8 \\
                   7 & 8 & 10 & 4 & 5 & 0 & 9 \\
                   10 & 11 & 13 & 7 & 8 & 9 & 0 \\
                 \end{array}
               \right)
$$
have the same spectrum.

\end{document}